\documentclass[a4paper,12pt]{article}

\usepackage{amsfonts}
\usepackage{amsmath}
\usepackage{amsthm}
\usepackage{hyperref}
\usepackage{graphicx}
\usepackage{color}
\usepackage{amssymb}
\usepackage{url}

\theoremstyle{plain}
\newtheorem{lemma}{Lemma}
\newtheorem{proposition}{Proposition}
\newtheorem{corollary}{Corollary}
\newtheorem{theorem}{Theorem}

\theoremstyle{definition}
\newtheorem{definition}{Definition}

\begin{document}

\title{\textbf{Relations between randomness deficiencies}}
\author{Gleb Novikov}
\date{}
\maketitle

\newcommand{\M}{M}
\newcommand{\A}{A}

\begin{abstract}
The notion of random sequence was introduced by Martin-L\"of  in \cite{ML}.
In the same article he defined the so-called randomness deficiency function
that shows how close are random sequences to non-random (in some natural sense).
Other deficiency functions can be obtained from the Levin-Schnorr theorem,
that describes randomness in terms of Kolmogorov complexity.
The difference between all of these deficiencies is bounded by a logarithmic term
(proposition \ref{dE<=dP}).
In this paper we show (theorems \ref{th1} and \ref{th2}) that the difference between some
deficiencies
can be as large as possible.
\end{abstract}

\section{Introduction}

Classical probability theory cannot deal with individual random
objects, such as binary sequences or
points on the real line: each  sequence or point has measure zero (with respect
to the uniform measure).
However our intuition says that the sequence of zeros (and any other computable sequence)
is not random, while the result of tossing a coin is random.
Martin-L\"of in \cite{ML} tried to formalize this statement. He used an algorithmic approach to
define random binary sequences.

Martin-L\"of  random sequences have many nice properties:
adding, deleting or changing finitely many bits doesn't change randomness;
random sequences satisfy the law of large numbers;
computable permutations  preserve randomness.
So if the sequence $\omega$ is random, the sequence $\omega^\prime = 0^{1000000000}\omega$
(billion of zeros concatenated with $\omega$) is also random. But intuitively $\omega^\prime$ is
``less random''.  We can make this arguement formal using a randomness deficiency
function $d$:
this function is finite on random sequences and infinite on non-random sequences.
If $d(\omega^\prime) \ge d(\omega)$ we say that $\omega^\prime$ is less random than $\omega$.
It turns out that there are some natural types of deficiency functions that have similar properties
to the so-called finite deficiency (the difference between the length of the string and its
Kolmogorov complexity). For example,
adding $n$ zeros to the sequence increases randomness deficiency by $n + O(\log n)$. Using this fact
one can reformulate statements about random sequences in terms of
the deficiency functions to look for the connections between algorithmic randomness and Kolmogorov
complexity theories.

In this paper we consider several deficiency functions: the first was introduced by
Martin-L\"of (definition \ref{ML def}), the others appear from the Levin-Schnorr's criterion of
randomness in
terms of different types of Kolmogorov complexity: the prefix-free complexity \eqref{eq:gacs} and
the a priori complexity (definition \ref{dA}). The difference between all of the deficiencies
is not greater than $(1 + \varepsilon)\log d$ (up to a constant, for all $\varepsilon > 0$)
(proposition \ref{dE<=dP}), where $d$ is one of the deficiency functions.
We show that the difference between some of the deficiencies can be greater than $\log d$.
For example, some of the deficiency functions (given in the exponential scale) are
integrable, while the others are not and that is the reason of the difference (theorem \ref{th1}).
To differ the integrable deficiencies we construct a special rarefied set of intervals
in the Cantor space (theorem \ref{th2}).
$\newline$

{\large{\textbf{Notation}}}

The set of all infinite binary sequences is called the Cantor space and is denoted by $\Omega$.
An interval in the Cantor space is a set of extensions of some string $x$, it is denoted by $[x]$.
The set of all binary strings is denoted by $\mathbb{B}^*$.
The length of the string $x$ is denoted by $|x|$.
We write $y\prec x$ if $y$ is a prefix of $x$.
$\mathbb{I}_{S}$ is the indicator function of the set $S$.
$\log$ means binary logarithm.
Notation $f<^+g\;$ ($f<^*g$) means that there exists a constant $c$ such that for
all $x\;$ $f(x) < c + g(x)\;$
($f(x)< c g(x)$).

\section{Preliminaries}

One can find all of the notions and statements of this section in \cite{classes}
and \cite{kolmbook}.

\begin{definition}
A measure $\mu$ over $\Omega$ is called computable,
if there exists a Turing machine that from each string $x$ and rational $\varepsilon > 0$
returns an $\varepsilon$-approximation of the value $\mu([x])$.
\end{definition}

The collection of intervals in the Cantor space forms a base for its standard topology.
We will talk about closed and open sets relative to this topology.

\begin{definition}
Let $\mu$ be a computable measure. A nested sequence of open sets $\{V_n\}$ is called a
Martin-L\"of test with respect to $\mu$ if:

1) $\{V_n\}$ is uniformly effectively open, that is there exists a Turing machine that for each
input $k$ enumerates the set $V_k$.

2) $\mu(V_n) \le 2^{-n}$ for each $n$.
\end{definition}

\begin{definition}\label{ML def} Let $\{V_n\}$ be a Martin-L\"of test with respect to a computable
measure $\mu$.
Function $d_{\mu;{\{V_n\}}}(\omega) = \max\{k : \omega \in V_k\}$ is called
a randomness deficiency of $\omega$ with respect to the test $\{V_n\}$.
\end{definition}

\begin{lemma}
For every computable measure $\mu$ there exists a Martin-L\"of test $\{U_n\}$ with respect to a
computable measure $\mu$ such that for any Martin-L\"of test $\{V_n\}$ with respect to
$\mu$ there exist a
constant $c$ such that for all sequences $\omega$
\begin{equation*}
d_{\mu;{\{U_n\}}}(\omega) \ge d_{\mu;{\{V_n\}}}(\omega) - c
\end{equation*}
\end{lemma}
\begin{proof}
We can enumerate all Martin-L\"of tests $\{U^j_n\}: U^j_1 \supset U^j_2 \supset \dotsm$ and
construct a new test:
\begin{align*}
U_1 = U^1_2 \cup U^2_3 \cup \ldots
\supset \dotsm \supset
U_2 = U^1_3 \cup U^2_4 \cup \ldots
\supset \dotsm \supset \\ \supset \dotsm \supset
U_n = U^1_{n+1} \cup U^2_{n+2} \cup \ldots
\supset \dotsm
\end{align*}
The new deficiency $\textbf{d}_\mu$ is not less than $d_{\mu;{\{U^j_n\}}} - j$.
\end{proof}

The deficiency function $\textbf{d}_\mu$ was defined by Martin-L\"of in \cite{ML}. In the same article he introduced the following notion of randomness:
\begin{definition}
Let $\mu$ be a computable measure. A sequence $\omega \in \Omega$ is called Martin-L\"of random
with respect to $\mu$ if $\textbf{d}_\mu(\omega) < \infty$.
\end{definition}

There are some other types of deficiency functions. To show the relations between them, we need to reformulate the definition of $\textbf{d}_\mu$.
First we define the so-called lower semicomputable functions.

\begin{definition}
A function $t : \Omega \rightarrow \mathbb{R}$ is called lower semicomputable if there exists a
machine that by rational $r$ enumerates the
set of intervals $\{\omega: t(\omega) > r\}$ (so this set should be open).
\end{definition}

Let's note the following property of $\textbf{d}_\mu$: the function $\textbf{t}_\mu =
2^{\textbf{d}_\mu}$ is probability bounded, that is
\begin{equation*}
\mu\{\textbf{t}_\mu(\omega) > c\} \le \frac{1}{c}
\end{equation*}
for rational numbers $c$.
Moreover, $\textbf{t}_\mu$ is the largest (up to a multiplicative constant) among
all lower semicomputable probability bounded functions (the sets $V_n = \{t(\omega) > 2^n\}$ form a Martin-L\"of test).
Therefore we can define the function $\textbf{d}_\mu$ as logarithm of the largest lower
semicomputable probability bounded function
and from now we denote this function as $\textbf{d}^P_\mu$ (and $\textbf{t}_\mu$ as
$\textbf{t}^P_\mu$).

To define other deficiency functions we need the following notion:

\begin{definition}
Function $f : \Omega \rightarrow \mathbb{Q}$ is called basic if its value on every sequence $\omega$
is determined by some finite prefix of $\omega$.
\end{definition}

By compactness of $\Omega$ there exist finitely many intervals where basic function is
constant, and the union of these intervals is $\Omega$.
Therefore basic functions are constructive objects and we can consider computable sequences of
basic functions.

The following lemma gives the equivalent definition of lower semicomputable functions.

\begin{lemma}
Function $t : \Omega \rightarrow \mathbb{R}$ is lower semicomputable iff it is a limit of increasing computable sequence of basic functions.
\end{lemma}

\begin{proof}
If the function $t$ is lower semicomputable then $t$ is a supremum of basic functions
$t_{n;k}(\omega) = n\mathbb{I}_{A_k}(\omega)$, where $A_k$ is a set of intervals produced after
$k$ steps of enumeration of $\{t_\mu(\omega) > n\}$.
Supremum is a limit of maximums and maximum over the finite set of basic functions is also a basic function.
If $t$ is a limit of increasing computable sequence of basic functions $t_n$ then for given $r$ we can produce intervals where $t_j > r$ for all $j$.
\end{proof}

If the function is integrable and its integral is less than $1$ it is probability bounded (by Markov's inequality). We call these functions expectation bounded.
There exists maximal (up to a multiplicative constant) lower semicomputable expectation bounded
function $\textbf{t}^E_\mu$:
we can enumerate all probability bounded functions (with respect to $\mu$); the integral of such
function is a limit of integrals of basic functions, so if it is greater than $1$ we always know it
after finitely many steps of computation. If the integral is greater than $1$, we decrease the
values of basic functions to make it less than $1$.
The sum of these new functions with weights $2^{-n}$ is the maximal lower semicomputable expectation bounded function.

\begin{definition}
Let $\mu$ be a computable measure. The expectation bounded deficiency is the function
\begin{equation*}
\textbf{d}^E_\mu(\omega) = \log \textbf{t}^E_\mu(\omega)
\end{equation*}
\end{definition}

The following proposition shows that the difference between $\textbf{d}^p_\mu$ and $\textbf{d}^E_\mu$ is not large.

\begin{proposition}\label{dE<=dP}
Let $\mu$ be a computable measure and $\varepsilon > 0$. Then
\begin{equation*}
\textbf{d}^E_\mu \le^+ \textbf{d}^P_\mu \le^+ \textbf{d}^E_\mu + (1+\varepsilon)\log\textbf{d}^E_\mu
\end{equation*}
\end{proposition}

\begin{proof}
The first part follows from Markov's inequality. To prove
the second part, let's consider a function
$\textbf{t}^P_\mu \log^{-1-\varepsilon}{\textbf{t}^P_\mu}$. Its integral does not exceed
\begin{equation*}
\sum_n \int_{A_n} \textbf{t}^P_\mu(\omega) \log^{-1-\varepsilon}{\textbf{t}^P_\mu(\omega)}
d\mu(\omega) \le
\sum_n2n^{-1-\varepsilon}
\end{equation*}
where
$A_n = \{2^n \le\textbf{t}^P_\mu < 2^{n+1}\}$, so this integral is finite.
Therefore
\begin{equation*}
\textbf{d}^P_\mu \le^{+} \textbf{d}^E_\mu + (1+\varepsilon)\log\textbf{d}^P_\mu \le^{+} \textbf{d}^E_\mu + (1+\varepsilon)\log\textbf{d}^E_\mu
\end{equation*}
\end{proof}

The deficiency function $\textbf{d}^E_\mu$ can be described in terms of prefix-free Kolmogorov complexity (see, for example, \cite{kolmbook}).
We will briefly describe this construction.
At first we define the discrete analogues of basic and lower semicomputable functions.

\begin{definition}
Function $f : \mathbb{B}^* \rightarrow \mathbb{Q}$ is called basic if its support is finite.
\end{definition}

\begin{definition}
Function $f : \mathbb{B}^* \rightarrow \mathbb{R}$ is called lower semicomputable if it is a limit of increasing computable sequence of basic functions.
\end{definition}

\begin{definition}
Lower semicomputable function $m : \mathbb{B}^* \rightarrow [0,\infty)$ such that $\sum_x{m(x)} \le 1$
is called discrete lower semicomputable semimeasure.
\end{definition}

Let's denote the prefix-free Kolmogorov complexity of a string $x$ as $K(x)$.
The function $\textbf{m}(x) = 2^{-K(x)}$ is called the discrete a priori probability. The famous coding theorem (see, for example, \cite{kolmbook}) states
that this function is the largest (up to a multiplicative constant) among all discrete lower semicomputable semimeasures.

It can be shown (see, for example, \cite{classes}) that

\begin{equation*}
\textbf{t}^E_\mu(\omega) =^{*}
\sum_n \frac{\textbf{m}(\omega_{1\ldots n})}{\mu([\omega_{1\ldots n}])} =^{*}
\sup_n \frac{\textbf{m}(\omega_{1\ldots n})}{\mu([\omega_{1\ldots n}])}
\end{equation*}

In the logarithmic scale:

\begin{equation}\label{eq:gacs}
\textbf{d}^E_\mu(\omega) =^{+} \sup_n \{-\log\mu([\omega_{1\ldots n}]) - K(\omega_{1\ldots n})\}
\end{equation}

This result is due to Gacs (see \cite{exact}).
The value in the right part of \ref{eq:gacs} is finite iff the sequence is random.
It was first shown by Schnorr and Levin independently in \cite{Schnorr} and \cite{Levin}.
Informally, the sequence is random iff its initial segments are incompressible.
The equation \ref{eq:gacs} also shows that if one adds $n$ zeros to the sequence then the randomness
deficiency (probability or expectation bounded)
increases by at most $n + O(\log n)$.

The Schnorr-Levin theorem can be formulated in terms of the so-called a priori complexity. To define
it we need the notion of continuous a  priori probability.

\begin{definition}
Lower semicomputable function $a : \mathbb{B}^* \rightarrow [0,\infty)$ such that $\sum_{x\in
S}{a(x)} \le 1$  for every prefix-free set $S$
is called continuous lower semicomputable semimeasure.
\end{definition}

We can enumerate all continuous lower semicomputable semimeasures and consider a semimeasure
$\textbf{a}(x) = \sum_ja_j(x) \textbf{m}(a_j)$.
This semimeasure is also continuous and lower semicomputable, and it is the largest (up to a multiplicative constant) in this class of semimeasures.
We will call $\textbf{a}(x)$ the continuous a priori probability.

\begin{definition}
The value $K\A(x) = -\log\textbf{a}(x)$ is called the a priori complexity of $x$.
\end{definition}

The Schnorr-Levin theorem for the a priori complexity states that the sequence $\omega$ is random
iff
$\sup_n\{-\log \mu([\omega_{1\ldots n}]) - K\A(\omega_{1\ldots n})\}$ is finite.
Moreover, supremum can be replaced by $\limsup$ or $\liminf$. Using this theorem we can define
other types of deficiency functions.

\begin{definition}\label{dA}
Let $\mu$ be a computable measure. We will consider functions
\begin{align*}
&\textbf{d}^\A_\mu(\omega) = \sup_n\{-\log \mu([\omega_{1\ldots n}]) -
K\A(\omega_{1\ldots n})\}
\\*
&\textbf{d}^{\limsup\A}_\mu(\omega) = \limsup_n\{-\log \mu([\omega_{1\ldots n}]) -
K\A(\omega_{1\ldots n})\}
\\*
&\textbf{d}^{\liminf\A}_\mu(\omega) = \liminf_n\{-\log \mu([\omega_{1\ldots n}]) -
K\A(\omega_{1\ldots n})\}
\end{align*}
and call them a priori randomness deficiencies.
\end{definition}

Each continuous lower semicomputable semimeasure can be represented as a probability distribution
on the initial segmets of outputs
of some probabilistic machine that prints bits one after another and does not have to stop (see,
for example, \cite{kolmbook}).
That is for each $a(x)$ there exists a machine $A$ such that
\begin{equation*}
a(x) = \mathbb{P}\{\text{the output of \textit{A} begins on the string \textit{x}}\}
\end{equation*}
Informally, the Schnorr--Levin theorem states that the sequence $\omega$ is random iff the probability of getting the initial segments $\omega_{1\ldots n}$
using a probabilistic machine cannot be much greater than getting it from a random generator (with
the distribution $\mu$).
The deficiency functions from the definition \ref{dA} show the difference between logarithms of these probabilities.

One can use supermartingales to define the deficiencies $\textbf{d}^\A_\mu$,
$\textbf{d}^{\limsup\A}_\mu(\omega)$, $\textbf{d}^{\liminf\A}_\mu(\omega)$.

\begin{definition}
Let $\mu$ be a measure on $\Omega$ and let $M$ be a function of binary strings.

If $\mu([x])M(x) = \mu([x0])M(x0) + \mu([x1])M(x1)$ the function $M$ is called a martingale.

If $ \mu([x])M(x) \ge \mu([x0])M(x0) + \mu([x1])M(x1)$ the function $M$ is called a
supermartingale.

If $ \mu([x])M(x) \le \mu([x0])M(x0) + \mu([x1])M(x1)$ the function $M$ is called a submartingale.
\end{definition}

If martingale (or sub/supermartingale) is not bounded on the initial segments of the sequence
$\omega$ we say that it wins on $\omega$.

If $\mu$ is computable, the supermartingale $\textbf{M}(x) = \frac{\textbf{a}(x)}{\mu([x])}$ is
the
largest (up to a multiplicative constant) among all lower semicomputable supermartingales.
Supermartingale $\textbf{M}(x)$ wins on all non-random sequences and does not win on random
sequences.

The deficiency $\textbf{d}^\A_\mu(\omega)$ is a supremum of $\textbf{M}(\omega_{1\ldots n})$,
the deficiencies
$\textbf{d}^{\limsup\A}_\mu(\omega)$ and $\textbf{d}^{\liminf\A}_\mu(\omega)$
are respectively
limsup and liminf of  $\textbf{M}(\omega_{1\ldots n})$.

Now we are going to show the relations between the deficiencies.

\begin{proposition}\label{dE<=dA}
\begin{equation*}
\textbf{d}^E_\mu \le^+ \textbf{d}^{\liminf\A}_\mu
\end{equation*}
\end{proposition}

\begin{proof}
We need to construct some continuous lower semicomputable semimeasure $a$. Once the approximation to
$\textbf{m}(x)$ increases by $\varepsilon$ we do the following:

1)increase the value of $a$ by $\varepsilon$ on prefixes of $x$

2)increase the value of $a$ by $\varepsilon \mu([y])/\mu([x])$ on the extensions $y$ of $x$.
$\newline$
If $\textbf{d}^E_\mu = R$ there exists a string $x$ such that
\begin{equation*}
-\log \mu([x]) - K(x) =^{+} R
\end{equation*}
and $\omega$ is the extension of $x$.
If $n > |x|$, the logarithm of $a$ is:
\begin{equation*}
\log a(\omega_{1\ldots n}) \ge -K(x) + \log \mu([\omega_{1\ldots n}]) - \log \mu([x])
\end{equation*}
Therefore
\begin{align*}
\textbf{d}^{\liminf\A}_\mu(\omega) &\ge^{+} \liminf_n\{-\log \mu([\omega_{1\ldots n}]) +
\log a(\omega_{1\ldots n}) \} \ge \\*
&\ge \liminf_n\{-\log \mu([x]) - K(x)\} = -\log \mu([x]) - K(x) =^{+} \textbf{d}^E_\mu
\end{align*}
The case $\textbf{d}^E_\mu = \infty$ can be considered in the same way.
\end{proof}

\begin{proposition}\label{dA<=dP}
\begin{equation*}
\textbf{d}^\A_\mu \le^+ \textbf{d}^{P}_\mu
\end{equation*}
\end{proposition}

\begin{proof}
It is sufficient to show that $\mu\{2^{\textbf{d}^\A_\mu}(\omega) > 2^c\} \le 2^{-c}$ for all rational $c$.
Let's fix $c$ and consider a set of strings
\begin{equation*}
S = \{x: \frac{\textbf{a}(x)}{\mu([x])} > 2^c,\; \forall y\prec x \;\;\frac{a(y)}{\mu([y])} \le
2^c\}
\end{equation*}
It is evident that $\omega \in \cup_{x\in S} [x]$ iff $\textbf{d}^\A_\mu(\omega) > c$. The set $S$
is prefix-free,
so
\begin{equation*}
\mu\{2^{\textbf{d}^\A_\mu}(\omega) > 2^c\} = \sum_{x \in S}{\mu([x])} < \sum_{x \in S}{\frac{a(x)}{2^{c}}} \le 2^{-c}
\end{equation*}
\end{proof}

Combining the results of Propositions \ref{dE<=dP}, \ref{dE<=dA} and \ref{dA<=dP} we can write down the following chain of inequalities:
\begin{equation*}\label{eq:chain}
\textbf{d}^E_\mu \le^+ \textbf{d}^{\liminf\A}_\mu \le^{+} \textbf{d}^{\limsup\A}_\mu \le^{+}
\textbf{d}^{\A}_\mu \le^+ \textbf{d}^P_\mu \le^+ \textbf{d}^E_\mu +
(1+\varepsilon)\log\textbf{d}^E_\mu
\end{equation*}

The natural question is about the difference between these deficiencies.

\section{New results}

Now we are going to show the relations between deficiency functions.
Proposition \ref{Doob} is an effective version of
Doob's martingale convergence theorem (see, for example, \cite{probability}) and
can be easily obtained from it.
Theorems \ref{th1} and \ref{th2} require lemma \ref{calculus}.
This lemma can be easily proved using standard techniques from calculus.

\begin{definition}
If the sequence $\omega$ is random relative to the oracle $0^\prime$ it is called $2$-random.
\end{definition}

\begin{proposition}\label{Doob}
Let $\mu$ be a computable measure.
If $\omega$ is $2$-random (with respect to $\mu$), then $\textbf{d}^{\limsup\A}_\mu(\omega) =
\textbf{d}^{\liminf\A}_\mu(\omega)$
\end{proposition}

\begin{proof}
Given rational numbers $\beta > \alpha > 0$ we can construct a $0^\prime$-computable
supermartingale $M^\beta_\alpha$ that wins on
sequences $\omega$ such that the supermartingale $\textbf{M}$ infinitely many times becomes smaller
than
$\alpha$ and greater than $\beta$
on the initial segments of $\omega$. Using the oracle we compute the values of $\textbf{M}$ and if
$\textbf{M}(x) < \alpha$
the values $M^\beta_\alpha(z)$ are equal to $\textbf{M}(z)$ on extensions $z$ of $x$ such that
$\textbf{M}(z) \le \beta$.
When we find extension $y$ such that $\textbf{M}(y) > \beta$ we just save the capital
($M^\beta_\alpha(yw) = M^\beta_\alpha(y)$) until we find some new
string $x$ with small $\textbf{M}(x)$. On the segments from $x$ to $y$ the value of $M^\beta_\alpha$ increases by $\frac{\beta}{\alpha}$ times.
The sum of all $M^\beta_\alpha$ with weights $\textbf{m}(\alpha, \beta)$ is a $0^\prime$-lower semicomputable supermartingale, so it is finite on
$2$-random sequences.
\end{proof}

\begin{corollary}
Let $\mu$ be a computable measure. Then $2^{\textbf{d}^{\limsup\A}_\mu}$ is the integrable function
with respect to $\mu$.
\end{corollary}

\begin{proof}
By Fatou's lemma:
\begin{equation*}
\int_\Omega\liminf_n\textbf{M}(\omega_{1\ldots n})d\mu(\omega) \le
\liminf_n\int_\Omega\textbf{M}(\omega_{1\ldots n})d\mu(\omega) = \liminf_n\sum_{|x|=n}\textbf{a}(x)
\le 1
\end{equation*}
$\textbf{d}^{\limsup\A}_\mu = \textbf{d}^{\liminf\A}_\mu$ almost everywhere, therefore
$2^{\textbf{d}^{\limsup\A}_\mu}$ is integrable.
\end{proof}

The greater deficiencies are not integrable (in the exponential scale). To show that
$2^{\textbf{d}^{\A}_\mu}$ is not integrable we need the following easy lemma from calculus:

\begin{lemma}\label{calculus}
If $c_k \ge 0$ and $\sum_{k=1}^\infty c_k < \infty$ and $R_k := \sum_{n=k+1}^\infty c_n > 0$,
then
\begin{equation*}
\sum_{k=1}^\infty \frac{c_k}{R_k \log\frac{1}{R_K}} = \infty
\end{equation*}
\end{lemma}

\begin{proof}
At first we will prove that
\begin{equation*}
\sum_{k=1}^\infty \frac{c_k}{R_k} = \infty
\end{equation*}
Denote $z_k = \frac{c_k}{R_k}$. It is evident that
\begin{equation*}
z_k = \frac{R_{k-1} - R_k}{R_k} = \frac{R_{k-1}}{R_k} -1
\end{equation*}
Therefore
\begin{equation*}
\frac{1}{R_k} = \frac{1}{R_0}\prod_{n=1}^k (1 + z_n)
\end{equation*}
If we take the logarithm from both parts, we get
\begin{equation}\label{eq:log}
\log\frac{1}{R_k} = \log\frac{1}{R_0} + \sum_{n=1}^k \log(1 + z_n) \le^{*} \sum_{n=1}^k z_n
\end{equation}
The left part tends to infinity, so the sum $\sum_{n=1}^\infty z_n$ is infinite.
To prove the lemma we need to show that $\sum_{k=1}^\infty \frac{z_k}{\log\frac{1}{R_k}} = \infty$.
Using \ref{eq:log} we get:
\begin{equation*}
\sum_{k=1}^\infty \frac{z_k}{\log\frac{1}{R_k}} \ge^{*} \sum_{k=1}^\infty \frac{z_k}{\sum_{n=1}^k z_n}
\end{equation*}
Denote $S_k = \sum_{n=1}^k z_n$ and $b_k = \frac{z_k}{S_k}$. It is sufficient to show that if the series $\sum_{n=1}^\infty z_n$ does not converge
then the series $\sum_{n=1}^\infty b_n$ also does not converge. We will do it in the same way as
the first part of the proof of the lemma:
\begin{equation*}
b_k = \frac{S_{k+1} - S_k}{S_k} = \frac{S_{k+1}}{S_k} - 1
\end{equation*}
Therefore
\begin{equation*}
S_{k+1} = S_1\prod_{n=1}^k (1 + b_n)
\end{equation*}
If we take the logarithm from both parts we get
\begin{equation*}
\log S_k = \log S_1 + \sum_{n=1}^k \log(1 + b_n) \le^{*} \sum_{n=1}^k b_n
\end{equation*}
The left part tends to infinity, so the sum $\sum_{n=1}^\infty b_n$ is infinite.
\end{proof}

Recall the definition of atomic measures.

\begin{definition}
If the measure $\mu$ on $\Omega$ is positive on some sequence, we will say that $\mu$ is an atomic
measure.
\end{definition}

Now we are ready to prove two statements about the difference between
$\textbf{d}^{\A}$ and other deficiencies.

\begin{theorem}\label{th1}
Let $\mu$ be a computable non-atomic measure. For all $c$ there exists $\omega$ such that
\begin{equation*}
\textbf{d}^{\limsup\A}_\mu(\omega) < \textbf{d}^{\A}_\mu(\omega) - \log\textbf{d}^{\A}_\mu(\omega) -
c
\end{equation*}
\end{theorem}

\begin{proof}
It is sufficient to prove that the function $q = 2^{\textbf{d}^{\A}_\mu - \log\textbf{d}^{\A}_\mu}$ is not integrable with respect to $\mu$.
We will construct some deterministic (but formally probabilistic) machine $f$. At each step, after $f$ has printed the string of bits $x$ of length $k$,
$f$ computes measures of $[x0]$ and $[x1]$, and then prints a bit $b$ if $\mu[xb] > \frac{1}{3}
\mu[x]$ (if the both bits are suitable, let $f$ print $0$).
Denote the interval $[xb] = B_k$ if at the $k$-th step $f$ prints a bit $b$, and $C_k = B_{k-1} -
B_k$. The measure $\mu$ is non-atomic, hence
\begin{equation*}
\mu B_k = \sum_{n=k+1}^\infty C_n
\end{equation*}
The intervals $C_k$ are disjoint, so $\sum_k C_k \le 1$. By lemma \ref{calculus}:
\begin{equation*}
\sum_{k=1}^\infty \frac{\mu C_k}{\mu B_k \log\frac{1}{\mu B_k}} = \infty
\end{equation*}
Let's denote
\begin{equation*}
a_f(x) = \mathbb{P}\{\text{the output of \textit{f} begins on the string \textit{x}}\}
\end{equation*}
and
\begin{equation*}
t_f(\omega) = \sup_n\frac{a_f(\omega_{1\ldots n})}{\mu([\omega_{1\ldots n}])}
\end{equation*}
The function $\frac{x}{\log x}$ is monotone for large enough $x$, therefore by the universality
\begin{equation*}
q \ge^{*} \frac{t_f}{\log t_f}
\end{equation*}
It is easy to see that
\begin{equation*}
\frac{t_f}{\log t_f}(\omega) = \sum_{k=1}^\infty \frac{\mathbb{I}_{C_{k+1}}}{\mu B_k \log\frac{1}{\mu B_k}} (\omega)
\end{equation*}
Recall that $\mu B_k \ge \mu B_{k+1} > \frac{1}{3} \mu B_k$
\begin{align*}
\int_\Omega q(\omega) d\omega \ge^{*} \int_\Omega\frac{t_f}{\log t_f} (\omega) d\omega  \ge \sum_{k=1}^\infty \frac{\mu C_{k+1}}{\mu B_k \log\frac{1}{\mu B_k}} >\\
> \frac{1}{3}\sum_{k=1}^\infty \frac{\mu C_{k+1}}{\mu B_{k+1} \log\frac{1}{\mu B_{k+1}}} = \infty
\end{align*}
\end{proof}

The next theorem requires some technical constructions in general case, so at first we will prove
it in the case of the uniform measure to show the idea.

\begin{theorem}\label{th2}
Let $\mu$ be a computable non-atomic measure. For all $c$ there exists $\omega$ such that
\begin{equation*}
\textbf{d}^{\A}_\mu(\omega) < \textbf{d}^{P}_\mu(\omega) - \log\textbf{d}^{P}_\mu(\omega) - c
\end{equation*}
\end{theorem}

\begin{proof}[Proof of the uniform case]

The main idea is that one cannot win $50$\$ after $5$ tosses of a coin if he starts with $1$\$.

Let's consider a function $g =  \sum_k 2^{2k - 1}\mathbb{I}_{[0^k1^k]}(\omega)$.
It is a lower semicomputable probability bounded function.
Let's prove the theorem by contradiction.
Assume that there exists a constant $c$ such that for all $\omega$
\begin{equation*}
\textbf{t}^{\A}_\mu(\omega) \ge 2^{-c}\frac{g}{\log{g}}(\omega)
\end{equation*}
That means that there exists a prefix-free set of binary strings $w^k_l$ such that
$\cup_l [w^k_l] \supset 0^k1^k$ and
\begin{equation*}
\textbf{a}(w^k_l)2^{|w^k_l|} \ge 2^{-c}\frac{2^{2k-1}}{2k-1}
\end{equation*}
For $k$ large enough
\begin{equation*}
|w^k_l| \ge -c - \log(2k - 1) + 2k - 1 + KA(w^k_l) > k + 1
\end{equation*}
So $[w^k_l] \subset [0^k1]$. Hence the set $\{w^k_l\}_{k,l}$ is prefix-free.
Consider the following chain of inequalities:
\begin{align*}
1 \ge \sum_k\sum_l \textbf{a}(w^k_l) \ge
\sum_k\sum_l 2^{-c - |w^k_l|}\frac{2^{2k-1}}{2k-1} \ge^{+} \\
\ge^{+}\sum_k 2^{-|0^k1^k|} \frac{2^{2k-1}}{2k-1} = \sum_k \frac{1}{2(2k - 1)} = \infty
\end{align*}
This contradiction proves the theorem.
\end{proof}

\begin{proof}[Proof of the general case]

Now we replace the intervals $[0^k1]$ and $[0^k1^k]$ by $C_k$ and $D_k$ (see below) respectively.
We cannot make the measures of $D_k$ very small, because it decreases $g$, but they also cannot be
large, because $g$ should be probability bounded. We will find suitable sets $\{C_k\}$ and $\{D_k\}$
that satisfy
all of the conditions.

Let's consider the intervals $B_k$ and $C_k$ from theorem \ref{th1}. The series $\sum\mu(C_k)$
is computable,
therefore the ordering $\tau$ of $\{C_k\}$ (the first element of the ordering has maximal measure
over $\{C_k\}$, the second has maximal measure over the rest of  $\{C_k\}$, etc.)
is also computable.
Denote the elements of this ordering by $\textbf{C}_k$ and
consider $z_k = -\frac{3}{\log \mu\textbf{C}_k}$.
The sequence $S_k = 1 + \sum_{j \le k} z_k$ is computable. Let's show that
$S_k \to \infty$:

Recall that
\begin{equation*}
\sum_k \frac{\mu C_{k+1}}{\mu B_k \log\frac{1}{\mu B_k}} = \infty
\end{equation*}
The function $\frac{x}{\log x}$ is monotone for large enough $x$, therefore
\begin{equation*}
\sum_k \frac{3}{\log{\frac{1}{\mu C_k}}} = 3\sum_k \frac{\mu C_k}{\mu C_k \log{\frac{1}{\mu C_k}}}
\ge 3\sum_k \frac{\mu C_{k+1}}{\mu B_k \log\frac{1}{\mu B_k}} = \infty
\end{equation*}

Now we are going to construct the set of intervals $D_k \subset C_k$ with such property:
\begin{equation}\label{eq:min}
\frac{1}{3}(\mu C_k)^{S_{\tau(k)}} < \mu D_k < (\mu C_k)^{S_{\tau(k)}}
\end{equation}
Let $x_k$ be a string such that $[x_k] = C_k$. We compute $\mu([x_k0])$ and $\mu([x_k1])$ and
choose the next bit $b$ if $\mu[x_kb] > \frac{1}{3} \mu[x_k]$
(if the both bits are suitable, let's choose $0$). After that we repeat this procedure with a
string $x_kb$ and so on.
We stop when the condition \ref{eq:min} holds for the interval $D_k$ (the set of the extensions of the latest string). It always happens, because the measure is
non-atomic (so $\mu[x_kb_1\ldots b_m]$ tends to $0$), and $\mu[x_kb_1\ldots b_{m-1}] < 3
\mu[x_kb_1\ldots b_m]$.

Consider a function
\begin{equation*}
g(\omega) = \sum_k \frac{\mathbb{I}_{D_k}(\omega)}{2 \mu D_k}
\end{equation*}
It is lower semicomputable. To prove that it is probability bounded it is sufficient to show that
\begin{equation*}
\mu D_j \ge \sum_{i : \mu D_i < \mu D_j} \mu D_i
\end{equation*}
Indeed, consider the set $\{g(\omega) > C\}$:
\begin{equation*}
\mu\{g(\omega) > C\} = \sum_{i: \mu D_i < \frac{1}{2C}} \mu D_i \le
2\max\{\mu D_i:  \mu D_i < \frac{1}{2C}\} < \frac{1}{C}
\end{equation*}
Consider the ordering $\pi$ of ${D_k}$ and denote the elements of this ordering by $\textbf{D}_k$.
The sequence $\mu\textbf{C}_j^{S_j}$ is exponentially decreasing:
\begin{equation*}
\frac{\mu\textbf{C}_j^{S_j}}{\mu\textbf{C}_{j+1}^{S_{j+1}}} \ge
\mu\textbf{C}_{j+1}^{S_j - S_{j+1}} =
\mu\textbf{C}_{j+1}^{-z_{j+1}} = 2^{ -z_{j+1} \log \textbf{C}_{j+1}} = 8
\end{equation*}
This inequality shows that $\textbf{D}_j \subset \textbf{C}_{j}$ (because
$\mu\textbf{D}_j >\frac{8}{3}\mu\textbf{C}_{i}^{S_i}$ if $i > j$) and moreover
\begin{equation*}
\sum_{i : \mu D_i < \mu D_j} \mu D_i = \sum_{l > \pi(j)} \mu\textbf{D}_l \le \sum_{k \ge 1}
(\frac{8}{3})^{-k} \mu\textbf{D}_{\pi(j)} <  \mu D_j
\end{equation*}
Therefore the function $g$ is probability bounded.

Assume that there exists a constant $c$ such that for all $\omega$
\begin{equation*}
\textbf{t}^{\A}_\mu(\omega) \ge 2^{-c}\frac{g}{\log{g}}(\omega)
\end{equation*}
Where $\textbf{t}^{\A}_\mu = 2^{\textbf{d}^{\A}_\mu(\omega)}$.
If $\omega \in D_k$, then for this $k$ there exists a prefix-free set of strings $w^k_l$ such that
$\cup_l [w^k_l] \supset D_k$ and
\begin{equation*}
\frac{\textbf{a}(w^k_l)}{\mu([w^k_l])} \ge 2^{-c}\frac{1}{2 \mu D_k \log{\frac{1}{\mu D_k}}}
\end{equation*}
Using the property \ref{eq:min} for large enough $k$ we get:
\begin{equation*}
\mu([w^k_l]) \le 2^{c+1} \textbf{a}(w^k_l) \mu D_k \log{\frac{1}{\mu D_k}} < \sqrt{\mu D_k} < \mu C_k
\end{equation*}
Therefore $w^k_l \subset C_k$ and the set $\{w^k_l\}_{k,l}$ is prefix-free.

Consider the following chain of inequalities:
\begin{align*}
1 \ge \sum_{k,l} \textbf{a}(w^k_l) \ge \sum_{k,l} \mu([w^k_l]) 2^{-c-1}\frac{1}{\mu D_k \log{\frac{1}{\mu D_k}}} \ge^{*} \\
\ge^{*} \sum_k \mu D_k  \frac{1}{\mu D_k \log{\frac{1}{\mu D_k}}} = \sum_k \frac{1}{\log{\frac{1}{\mu D_k}}} = \\
= \sum_k \frac{1}{\log{\frac{1}{\mu\textbf{D}_k}}}
=^{*} \sum_k\frac{1}{S_k\log{\frac{1}{\mu\textbf{C}_k}}}
\end{align*}
In the proof of lemma \ref{calculus} we showed that if the series $\sum_nz_n$ does not converge,
then the series $\frac{z_n}{S_n}$ where $S_n = \sum_{k\le n} z_k$
does not converge either, so the right part of the chain of inequalities is $\infty$.
\end{proof}

Now we can rewrite the chain of inequalities \ref{eq:chain} as follows:
\begin{equation*}
\textbf{d}^E_\mu \le^+ \textbf{d}^{\liminf\A}_\mu \;\overset{\mathrm{a.e.}}{=\joinrel=}\;
\textbf{d}^{\limsup\A}_\mu \;\ll\;
\textbf{d}^{\A}_\mu \;\ll\; \textbf{d}^P_\mu \le^+ \textbf{d}^E_\mu +
(1+\varepsilon)\log\textbf{d}^E_\mu
\end{equation*}
where the symbol $\ll$ means that the difference may be greater than
$\log\textbf{d}_\mu$.

One can ask a natural question about the difference between integrable (in the exponential scale)
deficiencies $\textbf{d}^E_\mu$ and $\textbf{d}^{\liminf\A}_\mu$ (or  $\textbf{d}^{\limsup\A}_\mu$).
We don't know the answer.


\newpage
\begin{thebibliography}{99}

\bibitem{classes}
Bienvenu L., Gacs P., Hoyrup M., Rojas C., and Shen A., Algorithmic tests and
randomness with respect to a class of measures, Proc. of the Steklov Institute of
Mathematics, v. 274 (2011), p. 41 -- 102.

\bibitem{kolmbook}
V.A.Uspensky, N.K.Vereshchagin, A.Shen, Kolmogorov complexity and algorithmic randomness,
MCCME, 2013 (in russian).

\bibitem{Introduction}
Ming Li and Paul M. B. Vitanyi. ´ Introduction to Kolmogorov Complexity
and its Applications (Third edition). Springer Verlag, New York, 2008.

\bibitem{ML}
Per Martin-Lof. The definition of random sequences. ¨ Information and
Control, 9:602 -- 619, 1966.

\bibitem{exact}
Peter Gacs. Exact expressions for some randomness tests. ´ Z. Math. Log.
Grdl. M., 26:385–394, 1980. Short version: Springer Lecture Notes in
Computer Science 67 (1979) 124 -- 131.

\bibitem{Schnorr}
Claus Peter Schnorr. Process complexity and effective random tests. J.
Comput. Syst. Sci, 7(4):376–388, 1973. Conference version: STOC 1972,
pp. 168 -- 176.

\bibitem{Levin}
Leonid A. Levin. On the notion of a random sequence. Soviet Math. Dokl.,
14(5):1413 -- 1416, 1973.

\bibitem{probability}
David Williams,
Probability with Martingales, Cambridge University Press (14 Feb. 1991)

\end{thebibliography}
\end{document}